\documentclass[12pt,onecolumn]{IEEEtran}

\everymath{\displaystyle}
\usepackage{graphics} 
\usepackage{epsfig} 
\usepackage{mathptmx} 
\usepackage{times} 
\usepackage{amsmath} 
\usepackage{dsfont}
\usepackage{epsfig}
\usepackage{amssymb} 
\usepackage{geometry}
\usepackage{xcolor}
\usepackage[normalem]{ulem}

\newcommand{\beqnum}{\begin{equation}\begin{array}{lcl}}
\newcommand{\eeqnum}{\end{array}\end{equation}}
\newcommand{\beqnom}{\begin{eqnarray}}
\newcommand{\eeqnom}{\end{eqnarray}}
\newcommand{\beqnc}{\begin{center}\begin{eqnarray}}
\newcommand{\eeqnc}{\end{eqnarray}\end{center}}
\newcommand{\beqnlm}{\begin{equation}\vspace{-.5cm}\begin{array}{lll}}
\newcommand{\eeqnlm}{\end{array}\end{equation}}\vspace{-.5cm}
\newcommand{\beq}{\begin{eqnarray*}}
\newcommand{\eeq}{\end{eqnarray*}}
\newcommand{\bb}{\bar{b}}
\newcommand{\bef}{\begin{figure}[tbh!]}
\newcommand{\enf}{\end{figure}}

\newcommand{\ede}{\hfill \rule [-.3em]{.5em}{.5em}}
\newcommand{\sign}{\hbox{sign}}

\usepackage{amsthm}

\newtheorem{theorem}{Theorem}[section]
\newtheorem{lemma}[theorem]{Lemma}
\newtheorem{proposition}[theorem]{Proposition}

\newtheorem{remark}[theorem]{Remark}

\title {Target-point based path following controller for a car-type vehicle using bounded controls
}

\author{Salah Laghrouche$^{*}$, Mohamed Harmouche$^{*}$  and Yacine Chitour$^{**}$\\
$^{*}$SET Laboratory, UTBM, Belfort, France.\\
$^{**}$L2S, Universit\'e Paris-Sud 11, CNRS, Gif-sur-Yvette, France.\\
\small{mohamed.harmouche@utbm.fr, salah.laghrouche@utbm.fr, yacine.chitour@lss.supelec.fr }\\
}

\date{}
\begin{document}
\maketitle

\section*{Abstract}                          
In this paper, we have studied the control problem of target-point based path following for car-type vehicles. This special path following task arises from the needs of vision based guidance systems, where a given target-point located ahead of the vehicle, in the visual range of the camera, must follow a specified path. A solution to this problem is developed through a non linear transformation of the path following problem into a reference trajectory tracking problem, by modeling the target point as a virtual vehicle. Bounded feedback laws must be then used on the real vehicle's angular acceleration and the virtual vehicle's velocity, to achieve stability. The resulting controller is globally asymptotically stable with respect and the proof is demonstrated using Lyapunov based arguments and a bootstrap argument. The effectiveness of this controller has been illustrated through simulations.

\section{Introduction}
In the field of autonomous vehicle guidance, navigation and control, path-following problem of car-type vehicles is of particular interest. Many contemporary researchers have published various techniques and strategies for this problem, such as \cite{SamsonAit,consolini1,Jiang,Samson,Egersted,Bianco_2004}. Among open-loop motion planning techniques, differential flatness approach has been significant in motion planning to drive vehicles on Cartesian paths \cite{Fliess,Rouchon}. In feedback control techniques, larger effort has been made on tracking problems. A backstepping approach has been presented in the context of tracking in \cite{Jiang2}. This approach has also been used in \cite{Changboon}, to develop a controller that is robust against vehicle skidding effects. Do et al. have further improved upon Jiang's backstepping method in \cite{K_DO2_2004} and \cite{K_DO1_2004}, by adding observers to render the controller output-feedback and extending it to tracking and stabilization for parking problems of a vehicle and introducing dynamic update laws to compensate for parametric uncertainty and modeling errors. In \cite{Aguiar} Aguiar et al. have used adaptive switched supervisory control combined with a non linear Lyapunov-based control to ensure the global convergence of the position tracking error to a small neighborhood of the origin.
Bloch and Drakunov \cite{Bloch1996} have used sliding mode control for the stabilization and tracking of a nonholonomic dynamic system. This controller is global and ensures convergence to the neighborhood of the desired trajectory. Lee et al. \cite{Lee2001} have proposed a saturated feedback controller for tracking of a unicycle-type vehicle, using its forward velocity and angular acceleration as control inputs. They have also extended this controller for application on car-type vehicles.

The problem of path following differs from pure stabilization or tracking problems because the path, described by its curvature $\kappa(.)$, is defined in space only, not in time. In this paper, we have addressed the path following control of a robot car-type vehicle using target point.  This control problem arises from camera-vision applications \cite{broggi,Piazzi1}, where the vehicle is guided by a target point ahead of the vehicle, within the visual range of the camera \cite{Bertozzi1,broggi}. The target point is fixed at a known distance $d>0$ from the center of gravity on the axis of the vehicle. The control objective is to drive the vehicle, such that the target point follows the desired path (as shown in Fig. 1). This problem has been addressed in \cite{consolini2} where a local robust path following strategy has been proposed using target point. Their solution is based on an open loop control based on inversion of the nominal model, and a closed loop control for stabilization of the resultant system. The error dynamics have been expressed in the Fr\'{e}net frame associated to the followed path. This technique, also discussed in \cite{MS}, is convenient only when the vehicle is close, positioned and oriented to the path.

In our work, a global asymptotically stable controller is developed by parameterizing the path as a ``virtual vehicle'', which is tracked by the actual vehicle. In this way, the path following problem is converted into a tracking problem, with two control inputs: the angular acceleration of the real vehicle and the velocity of the virtual vehicle. The forward velocity control of the real vehicle is not considered as part of the navigation problem, as it is controlled by other intelligent control systems in practical applications (for example, ABS, ESP \cite{pasillias}). It is instead assumed to be a measured state that is  strictly positive, meaning that the vehicle is in continuous forward motion.

It can be noted that if there is no target point, i.e. $d=0$, then the tracking error model obtained in this study is identical to \cite{Lee2001}, in which tracking has been achieved by using saturation on one control input while the other is unbounded. In our case, the introduction of the target point at a distance makes the dynamics of the tracking error model more complicated. Specifically, the development produces a first order nonlinear non-globally Lipschitz differential equation (see equation~\eqref{omega}) that can blow up in finite time. To overcome this difficulty, our solution necessitates the application of saturated controls for both our control inputs with arbitrary small amplitude. Examples of application of saturated control can be found in \cite{Chitour1995,Liu1996,Yakoubi2006,Yakoubi2007,Yakoubi2007-1,Yakoubi2007-2}. Consequently, if both the control inputs are applied on the real vehicle, then the path following problem developed here becomes equivalent to the generalization of \cite{Lee2001}, as tracking problem with a target-point. 

This paper can be seen as the continuation of the authors' previous work in \cite{ait}, in which a unicycle type vehicle had been considered. However, the arguments of the Lyapunov analysis used for the convergence proof are significantly more involved than that of \cite{ait}, due to the added state of the car type vehicle (essentially an integrator) and the fact that one must keep track of the small amplitudes of the saturations. Therefore, a positive definite function $V$ is designed instead of a global Lyapunov function, whose time derivative along the closed-loop system is strictly negative outside a neighborhood of the origin. The design of $V$ relies on an asymptotic analysis of a Ricatti equation, which is not needed in \cite{ait}. The convergence to zero is then demonstrated using a bootstrap procedure \cite{bootstrap}, i.e., once the system errors converge to a neighborhood, they continue to diminish to a smaller neighborhood, and ultimately converge asymptotically to the origin. The results so obtained can be extended to the case where only the position of the reference trajectory is directly known.

The paper is organized as follows: in Section 2, the vehicle model and reference trajectory parameterization have been presented. In Section 3, the control design has been discussed and the Lyapunov function has been developed. The stability analysis of the closed loop system has been discussed in Section 4. Simulation results have been provided in Section 5, and a conclusion has been presented in Section 6.

\section{Vehicle model and reference trajectory}
Let us consider a path $\Gamma$ with geodesic curvature $\kappa_r$ whose absolute value is bounded by $\kappa_{max}>0$, i.e.,
for all $t\geq 0$, we have
\begin{equation}\label{kappa0}
|\kappa_r(t)|\leq \kappa_{max}.
\end{equation}
As described in the introduction, $\Gamma$ will be parameterized as a vehicle trajectory with a forward velocity $u(t)$ such that $\Gamma(t)=(p_r(t),q_r(t))$ is described by the following state equations
\begin{eqnarray}
\left \{ \begin{array}{rcl}
\dot{p}_r&=&u\  \cos{\psi_r},\\
\dot{q}_r&=&u\  \sin{\psi_r},\\
\dot{\psi}_r&=&u\  \kappa_r,\\
\dot{\kappa}_r&=&u\  \rho_r,
\end{array} \right.
\end{eqnarray}
where $\psi_r$ represents the angle between the abscissa axis and the velocity vector $(\dot p_r,\dot q_r)^T$, and
$\kappa_r$ is the scalar curvature associated to the parametrization of $\Gamma$ by time $t$. The arclength $s$ of $\Gamma$ is given by $s(t)=s_0+\int_0^tu(\tau)d\tau$ and the scalar curvature $\kappa_r(t)$ is hence equal to  $\kappa_r^*(s(t))$.

The state equations for the vehicle dynamics are 
\begin{eqnarray}\label{veh1}
\left \{ \begin{array}{rcl}
\dot{x}&=&V_x\ \cos{\psi},\\
\dot{y}&=&V_x\ \sin{\psi},\\
\dot{\psi}&=&V_x\  \kappa,\\
\dot{\kappa}&=&V_x\  \rho_0.
\end{array} \right.
\end{eqnarray}
These equations represent the vehicle's motion with a velocity $V_x$, along the curve defined by the its geodesic curvature $\kappa$. The control variable $\rho_0$ will be defined later.
Notice that $V_x$ is not necessarily constant, but simply a continuous function of time, which verifies the following hypothesis: there exist two positive constants $0<V_{min}\leq V_{max}$, such that for all $t\geq 0$
\begin{equation}\label{vx0}
V_{min}\leq V_x(t)\leq V_{max}.
\end{equation}
The strict positivity of the lower bound is necessary to derive our subsequent results. Note that path following for a unicycle type of vehicle has been obtained under weaker hypotheses than that of the above equation, cf. in particular the {\it persistent excitation} condition (PEC) \cite{Khalil}. It is not clear to us how to extend the present work only assuming that $V_x$ satisfies the PEC (see Remark \eqref{pec:0}).

For the target point, the equations for the coordinates $p$ and $q$ are defined as
\begin{eqnarray}
\begin{array}{rcl}
p&=x+ d \cos{\psi},\\
q&=y+ d \sin{\psi}.\\
\end{array}
\end{eqnarray}
We will also suppose throughout the paper that
\begin{itemize}
\item [{\textbf (H1)}] $d\kappa_{max}<1$.
\end{itemize}
\begin{remark} The above assumption may be considered as a technical one or a design constraint for positioning the target point. However, it is reasonable to upper bound  the curvature of the reference path in terms of the distance $d$. Indeed,
tracking a circle of radius $d'<d$ with a point fixed at a distance $d$ in front of a vehicle is impossible. To see that, one can see that intuitively of rely on equation \eqref{eq:kap2} given below. At the ligth of the previous example, Hypothesis $(H1)$ is almost optimal.
\end{remark}

The dynamics of the target point in a form similar to \eqref{veh1} can be obtained by deriving the precedent equations. One first gets that
\begin{eqnarray}
\left \{ \begin{array}{rcl}
\dot{p}&=&V_x\ \cos{\psi}  - d\ V_x\ \sin{\psi}\ \kappa, \\
\dot{q}&=&V_x\ \sin{\psi} + d\ V_x\ \cos{\psi}\ \kappa, \\
\end{array} \right.
\end{eqnarray}

\begin{figure}[h]
\begin{center}
\medskip
  \includegraphics[width=10cm]{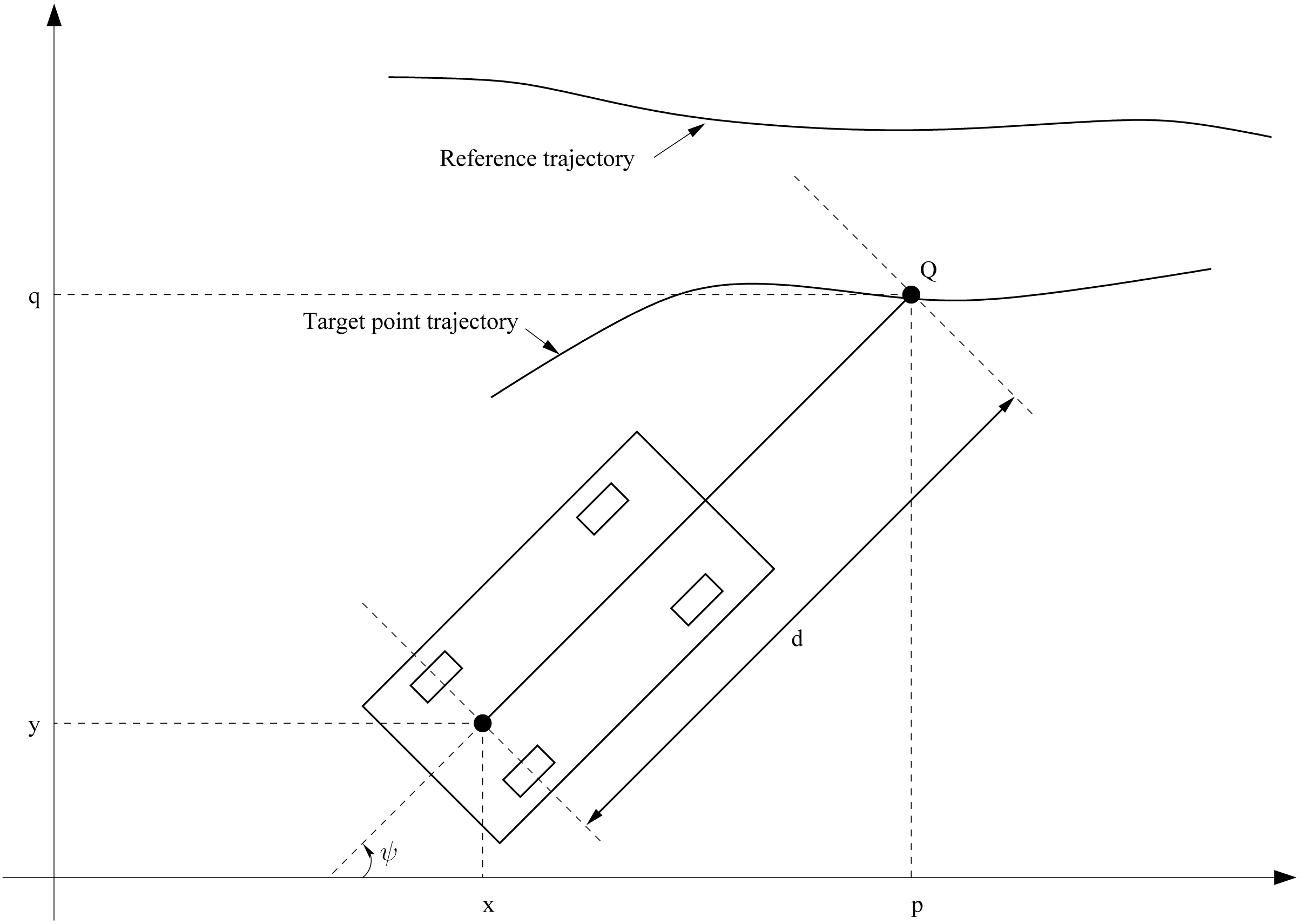}
\smallskip
\end{center}
  \caption{The reference trajectory, the vehicle and its target point.}
\end{figure}

The curve defined by the target point is traveled at the following speed
\begin{equation}\label{eq:vd}
v_d= \sqrt{\dot p^2+\dot q^2}=V_x \sqrt{1+(\kappa d)^2}.
\end{equation}

Our objective now is to define the dynamics of the target point as those of a car. For that purpose, we consider $\theta$ as the angle between the abscissa axis and the velocity vector $(\dot p,\dot q)^T$. One easily gets that
\begin{equation}\label{theta}
\theta =\psi +\arctan(\kappa d),
\end{equation}
then $\dot p= v_d\cos(\theta),\ \ \ \dot q=v_d\sin(\theta)$,
and the scalar curvature $\omega$ is defined by
${\displaystyle\omega=\frac{\dot{\theta}}{v_d}}$.

We define the dynamics of $\omega$ by the new control variable $\rho:=\dot{\omega} / v_d\ $. Deriving equation~\eqref{theta}, we obtain
\begin{equation}\label{omega}
d\dot\kappa=V_x(1+(\kappa d)^2)((1+(\kappa d)^2)^{1/2}\omega-\kappa).
\end{equation}

 Hence the dynamics of the target point $(p,q)$ becomes
\begin{eqnarray}\label{toto}
\left \{ \begin{array}{rcl}
\dot{p}&=&v_d\ \cos{\theta},\\
\dot{q}&=&v_d\ \sin{\theta},\\
\dot{\theta}&=&v_d\  \omega,\\
\dot{\omega}&=&v_d\  \rho.
\end{array} \right.
\end{eqnarray}
The error between the target point and the reference curve is defined as
\begin{eqnarray}
\begin{array}{rcl}
e_p&=&p-p_r,\\
e_q&=&q-q_r,\\
\xi&=&\theta - \psi_r,\\
\eta&=&\omega - \kappa_r.
\end{array}
\end{eqnarray}
and the error dynamics is given by
\begin{eqnarray}\label{erreur}
\left \{ \begin{array}{rcl}
\dot{e}_p&=&v_d\ \cos{\theta} - u\cos{\psi_r},\\
\dot{e}_q&=&v_d\ \sin{\theta} - u\sin{\psi_r},\\
\dot{\xi}&=&v_d\  \omega - \kappa_r\ u,\\
\dot{\eta}&=&v_d\  \rho - \rho_r\ u.
\end{array} \right.
\end{eqnarray}

\section{Control design and Lyapunov Function}
In this section, we will present a control law $u(e_p,e_q,\xi,\eta,t)$ and $\rho(e_p,e_q,\xi,\eta,t)$, such that the system (\ref{erreur}) is globally asymptotically stable (GAS for short) w.r.t. origin. Note that, from the equations \eqref{eq:vd} and \eqref{omega}, one recovers the control $\rho_0$ once $v_d$ and $\rho$ are determined. However, there is an issue of possible blow-up in finite time for $\kappa$ (and thus for $\rho_0$). Indeed, assuming that one is able to stabilize \eqref{erreur} to zero, then the control $\rho_0$ is obtained by derivating $\kappa$, which is in turn obtained by solving \eqref{omega}, seen as an o.d.e. with unknown $\kappa$ since $\omega$ tends to $\kappa_r$ asymptotically. Equation \eqref{omega} is of the type $\dot\kappa=f(\kappa,t)$ with the right-hand side $f$ not globally Lipschitz w.r.t. $\kappa$, hence it is not immediate to insure global existence of $\kappa$ for all $t\geq 0$. We will show later on, that an appropriate choice of $u$ and $\rho$ under Hypothesis {\textbf (H1)} solves this problem (see Lemma~\ref{explosion} below).

The standard saturation function $\sigma(x)$ defined for $x\in\mathbb{R}$ by
\begin{equation}\label{sigma}
\sigma (x) =  \frac{x}{\max(1,\vert x\vert)}.
\end{equation}
Let us first of all perform a variable change on the control, as follows
\begin{eqnarray}\label{com0}
\begin{array}{rcl}
u&=&v_d(1+u_1),\\
\rho&=&\rho_r(1+u_1) + u_2,
\end{array}
\end{eqnarray}
where $u_1,\ u_2$ are the new control variables.

\begin{remark}\label{pec:0} In order to define $\omega$, $\rho$ and to perform the change of inputs variables, $v_d$ must be greater than zero and thus $V_x$ must also be strictly positive. It is therefore not obvious to proceed as above, if $V_x$ only satisfies (PEC).
\end{remark}

With the boundedness of $\kappa$ and $V_x$, equation \eqref{eq:vd} implies that $\nu_d$ is bounded. If one insists on having $\rho$ bounded, then we must assume also that $\rho_r$ is bounded, as
\beqnum
	\left| \rho_r \right| \le \rho_{r,max},
\eeqnum
where $\rho_{r,max}$ is a known positive constant.\\

The system (\ref{erreur}) is therefore rewritten as
\begin{eqnarray}\label{err1}
\left \{ \begin{array}{rcl}
\dot{e}_p&=&v_d(\cos{\theta}- \cos{\psi_r}-\ u_1\cos{\psi_r}),\\
\dot{e}_q&=&v_d(\sin{\theta} - \sin{\psi_r}-\ u_1\sin{\psi_r}),\\
\dot{\xi}&=&v_d(\eta - \kappa_r u_1),\\
\dot{\eta}&=&v_d u_2.
\end{array} \right.
\end{eqnarray}

The bounded controls $u_1$ and $u_2$ can be expressed in the following form:
\begin{eqnarray}
\begin{array}{rcl}
u_1&=&C_1\sigma(\cdot),\\
u_2&=&D \sigma(\cdot),
\end{array}
\end{eqnarray}
with sufficiently small gains $C_1$ and $D$. Since $\kappa$ is bounded, $v_d$ also remains uniformly bounded throughout $t\geq 0$. We can hence change the time scale by considering $ds=v_d\ dt$. To keep the notations simple, we would continue to use $t$ for time, and the point for the derivation with respect to $s$, like $\frac{df}{ds}=\dot f $. This has no effect on the control laws since our design is based on static feedback (w.r.t. the error). The error dynamics hence becomes
\begin{eqnarray}
\left \{ \begin{array}{rcl}
\dot{e}_p&=&\cos{\theta}- \cos{\psi_r}-u_1\cos{\psi_r}, \\
\dot{e}_q&=&\sin{\theta} - \sin{\psi_r}-u_1\sin{\psi_r}, \\
\dot{\xi}&=&\eta - \kappa_r u_1,\\
\dot{\eta}&=&u_2.
\end{array} \right.
\end{eqnarray}

Let us perform the following change of variable corresponding to a time-varying
rotation in the frame of the reference trajectory
\begin{eqnarray}
\begin{array}{rcl}
y_1&=&e_p\ \cos{\psi_r} + e_q\ \sin{\psi_r}, \\
y_2&=&-e_p\ \sin{\psi_r} + e_q\ \cos{\psi_r}.
\end{array}
\end{eqnarray}
The system becomes
\begin{eqnarray}\label{systemu}
\left\{ {\begin{array}{rcl}
\dot{y}_1&=&-u_1 + (\cos{\xi}-1) + (1+u_1) \kappa_r y_2,\\
\dot{y}_2&=&\sin{\xi} - (1+u_1) \kappa_r y_1,\\
\dot{\xi}&=&\eta - \kappa_r u_1,\\
\dot{\eta}&=&u_2,
\end{array} }\right.
\end{eqnarray}
where $u_1$, $u_2$ will be chosen such that ($\ref{systemu}$) becomes GAS.\\ The control variables $u_1$ and $u_2$ are defined as follows
\begin{eqnarray}
\begin{array}{rcl}\label{control0}
   {u_1} &=&  {C_1}\sigma(y_1) ,\\
   {u_2} &=& -D\sigma( \frac{k_1}D\xi  + \frac{k_2}D\eta  + \frac{C_2}D\sigma (y_2)) ,\hfill \\
\end{array}
\end{eqnarray}
where $k_1,k_2,C_1,C_2,D$ are positive real numbers and $\sigma()$ is the standard saturation function defined in ($\ref{sigma}$).  Typically, we want to stabilize the system with arbitrarily small saturation levels $C_1$ and $D$. In conclusion, the final system, noted ($\Sigma$) becomes

\begin{eqnarray}\label{system}
(\Sigma) \quad \left  \{ {\begin{array}{rcl}
\dot{y}_1&=&-{C_1}\sigma  ({{y_1}}) + (\cos{\xi}-1) + \mu y_2,\\
\dot{y}_2&=&\sin{\xi} - \mu y_1,\\
\dot{\xi}&=&\eta - \kappa_r {C_1}\sigma ({{y_1}}),\\
\dot{\eta}&=&-D\sigma( \frac{k_1}D\xi  + \frac{k_2}D\eta  + \frac{C_2}D\sigma (y_2)) ,
\end{array}}\right.
\end{eqnarray}
where \begin{equation}\label{eq:mu0}
\mu:=\kappa_r (1 + {C_1}\sigma  ({{y_1}}) )\hbox{ and }
\vert\mu\vert\leq \kappa_{max}(1+C_1).
\end{equation}
In the following section, it is shown that Global Asymptotic Stability of the system (\ref{system}) can be achieved by proper selection of $C_1$, $C_2$, $k_1$, $k_2$.

More precisely, we prove the following theorem, which is the main result of the paper.

\begin{theorem}\label{theo1}
Consider a path $\Gamma$ with geodesic curvature $\kappa_r^\ast$ verifying \eqref{kappa0} for some $\kappa_{max}>0$. It is then possible to track asymptotically $\gamma$ with a point fixed at a distance $d>0$ in front of a vehicle, where $d\kappa_{max}<1$, by choosing the control laws
$u_1,u_2$ according to \eqref{control0} with constants $k_1,k_2,C_1,C_2,D$, which satisfy the following conditions. Set $a:=\frac{3}{16}$.
\begin{equation}\label{eq:param}
k_1 = a{k_2}^2,\ C_2=\frac1{2\beta k_2},\ C_1=\frac{aC_2}{4k_2},
\end{equation}
where $\beta$ is a positive constant larger than $8$, $D$ is an arbitrary positive constant, fixed a-priori, and $k_2$ is large enough that $\frac1{k_2D} \ll 1$.

\end{theorem}


\begin{proof}
The proof of GAS stability of System (\ref{system}) has been carried out as an argument based on Lyapunov analysis.\\

The first remark consists in focusing on the last two equations in $(\Sigma)$
and we will first treat the case where there is no saturation on $\dot\eta$.

In that case, the last two equations in the previous section define a double integrator system, which shall now be denoted as ($S_k$):
\begin{eqnarray}\label{integrators}
\begin{array}{rcl}
(S_k) \quad \left \{ \begin{gathered}
  \dot \xi  = \eta  + {v_1} ,\hfill \\
  \dot \eta  =  - {k_1}\xi  - {k_2}\eta  + {v_2} ,\hfill \\
\end{gathered}  \right.
\end{array}
\end{eqnarray}
with,
\begin{eqnarray}\label{eq:nu00}
\begin{array}{rcl}
\begin{gathered}
{v_1} =  - {\kappa _r}{C_1}\sigma \left( {{y_1}} \right), \hfill \\
{v_2} =  - {C_2}\sigma ({y_2}).\hfill \\
\end{gathered}
\end{array}
\end{eqnarray}

%

The system $(S_k)$ can be presented in the matrix form
\begin{eqnarray}\label{matrix_form}
\begin{array}{rcl}
   \dot Z = AZ + BU ,
\end{array}
\end{eqnarray}
where,
\begin{eqnarray}
\begin{array}{l}
 Z = \left( \begin{array}{l}
 \xi  \\
 \eta  \\
 \end{array} \right), \
 A = \left( {\begin{array}{*{20}{c}}
   0 & 1  \\
   { - {k_1}} & { - {k_2}}  \\
\end{array}} \right), \
 B = \left( {\begin{array}{*{20}{c}}
   1 & 0  \\
   0 & 1  \\
\end{array}} \right), \
 U = \left( \begin{array}{l}
 {\nu _1} \\
 {\nu _2} \\
 \end{array} \right). \\
 \end{array}
 \end{eqnarray}
Since $A$ is Hurwitz, there exists a quadratic form $V_k = Z^T P_k Z $, where $P_k$ is a positive definite square matrix, obtained by solving the following Riccati equation
\begin{eqnarray}\label{eq:ricatti}
\begin{array}{rcl}
{P_k}A + {A^T}{P_k}+\frac{P_k^2}{\Upsilon _L^2} =  - I,
\end{array}
\end{eqnarray}
where, $\Upsilon _L$ is the $L_2$-gain related to the system $(S_k)$.
The derivative $\dot V_k$ is given by the following equation\\
\begin{eqnarray}\label{eq:V}
\begin{array}{rcl}
{{\dot V}_k} =  - {\left\| Z \right\|^2} - \frac{{{{\left\| {P_k} Z \right\|}^2}}}{{\Upsilon _L^2}} + 2{Z^T}{P_k}U,
\end{array}
\end{eqnarray}
and verifies
\begin{eqnarray}
\begin{array}{rcl}
{{\dot V}_k} \leq  - {\left\| Z \right\|^2} +\Upsilon _L^2{\left\| U \right\|^2}.
\end{array}
\end{eqnarray}
The Lyapunov function proposed for the global system (\ref{system}) is
\begin{eqnarray}\label{Lyapunov}
\begin{array}{rcl}
  V&=&M{V_k} + {k_1}\frac{{ y_1^2 + y_2^2 }}{2} + \eta {y_2} + {k_2}{y_2}\xi, \\
\end{array}
\end{eqnarray}
where $M,\ k_1,\ k_2$ are positive constants to be chosen later in particular to ensure
that V is positive definite function, see Lemma \ref{le:V} below. Moreover, a straightforward computation yields the following:
\begin{proposition}
The derivative of the Lyapunov function can be upper bounded as follows,
\begin{eqnarray}
\begin{array}{rcl}
\begin{gathered}
\begin{array}{rcl}
 \dot V  &\leq&  - M({\xi^2 + \eta^2}) - {k_1}{C_1}{y_1}\sigma \left( {{y_1}} \right) - {C_2 y_2}\sigma \left( {{y_2}} \right) + M{\Upsilon _L}^2\left( {{{\left( {{\kappa _r} C_1 \sigma \left( {{y_1}} \right)} \right)}^2} + ({C_2 \sigma \left( {{y_2}} \right)}) ^ 2} \right) \\
  &&+k_1y_2(\sin\xi-\xi)+k_1y_1(\cos\xi-1) - \mu \eta {y_1} + \eta \sin\xi  - {k_2}{y_2}{\kappa _r}{C_1}\sigma \left( {{y_1}} \right)
  - {k_2}\xi \mu {y_1} + {k_2}\xi\sin\xi. \label{lyap1}\\
 \end{array}
\end{gathered}
\end{array}
\end{eqnarray}
\end{proposition}

The rest of the argument is divided in two main steps. In the first step, the existence of appropriate constants $M,k_1,k_2,C_1,C_2$ is proven, such that $V$ has a positive definite quadratic form in all the variables. This means that there exists a bounded region $Y_{k_2}$ (for $k_2$ typically large), in the $(y_1,y_2)-$plane:

 \begin{equation}\label{eq:Yk}
 Y_{k_2}=\{(y_1,y_2)\vert \quad \vert y_1\vert\leq \frac{C}{k_2^2},\ \vert y_2\vert\leq \frac{C}{k_2^{3/2}}\},
 \end{equation}

such that outside this region, the derivative of $V$ along trajectories of \eqref{system} fulfills the following inequality

\begin{equation}\label{final}
 \dot V  \leq  - \frac{M}2({\xi^2 + \eta^2}) - {k_1}\frac{C_1}2{y_1}\sigma \left( {{y_1}} \right) - \frac{C_2}2{y_2}\sigma \left( {{y_2}} \right).
\end{equation}

In the second step, a bootstrap-type argument is applied to show the convergence of trajectories of \eqref{system} to zero, as $t$ tends to infinity.\\

These two steps have been achieved in the following manner: the $L_2$-gain of $(S_k)$, denoted by $\Upsilon_L$ is calculated, then $P_k$ is estimated for $k_2$ tending to infinity. Then ISS (input-to-state) type bounds are calculated for $\xi$ and $\eta$ and the derivative of the Lyapunov function is estimated outside $Y_{k_2}$, and the argument is concluded. The detailed calculations have been presented in the following subsections.


\subsection{$L_2$-gain  $\Upsilon_L$}
Let us study the system ($S_k$), defined in the equation (\ref{integrators}).
We recall that, ($S_k$) can be presented in the following matrix form

\begin{eqnarray}\label{matrix_form_ex}
\begin{array}{rcl}
\underbrace {\left( {\begin{array}{*{20}{c}}
   {\dot \xi }  \\
   {\dot \eta }  \\
\end{array}} \right)}_{\dot Z} = \underbrace {\left( {\begin{array}{*{20}{c}}
   0 & 1  \\
   { - {k_1}} & { - {k_2}}  \\
\end{array}} \right)}_A\underbrace {\left( {\begin{array}{*{20}{c}}
   \xi   \\
   \eta   \\
\end{array}} \right)}_Z + \underbrace {\left( {\begin{array}{*{20}{c}}
   1 & 0  \\
   0 & 1  \\
\end{array}} \right)}_B\underbrace {\left( {\begin{array}{*{20}{c}}
   {{v_1}}  \\
   {{v_2}}  \\
\end{array}} \right) . }_U
\end{array}
\end{eqnarray}

\begin{lemma}\label{le:l2}
We will tune $k_2 \geq 20$ with $k_1=\frac{3}{{16}} k_2^2$, then $1 < \Upsilon_L < 1.2$.
\end{lemma}
The proof of Lemma \ref{le:l2} is given in Appendix.


\subsection{Estimation of $P_k$ for large $k_2$}
In this section, we take ${k_1} = ak_2^2$ with $a=\frac{3}{{16}}$ and $k_2$ tending to infinity.
We will prove the following two results whose proofs are given in Appendix.

\begin{lemma}\label{le:P_k}
As $k_2$ tends to infinity, the positive define matrix $P_k$ defined in \eqref{eq:ricatti} admits the following asymptotic expansion
\begin{equation}\label{eq:P}
P_k=\begin{pmatrix}p_1k_2&p_2\\p_2&\frac{p_3}{k_2}\end{pmatrix},
\end{equation}
with $p_i=F_i(1+O(\frac1{k_2^2}))$, $1\leq i\leq 3$, where the $F_i$ are positive constants
(only depending on $a$) so that $F_1F_3-F_2^2>0$.
\end{lemma}

\begin{proposition}\label{le:V}
For $k_2$ large enough and $M = \beta k_2$, the function $V$ defined in \eqref{Lyapunov} is a positive quadratic form in $(\xi,\eta,y_1,y_2)$.
\end{proposition}
\subsection{ISS bounds for $\xi$ and $\eta$ }
For a real-valued continuous and bounded function $f$ defined on $\mathbb{R}^+$,
we set
$$
\vert f\vert^{*}(t):=\sup_{s\geq t}\vert f(s)\vert,
$$
and
$$
\Vert f\Vert^*:=\limsup_{s\rightarrow\infty}\vert f(s)\vert.
$$
\begin{lemma}\label{l_sup_limite_non_satur}
Consider the system (\ref{matrix_form_ex}). By tuning ${k_1} = \frac{3}{{16}}k_2^2$ ,
the ISS bounds of $\xi$ and $\eta$ satisfy the following inequalities for $t\geq 0$,
\begin{eqnarray}\label{ISS_limite}
\begin{array}{rcl}
\begin{gathered}
\left\{ \begin{array}{l}
 {\begin{array}{*{20}{c}}
\vert \xi\vert^{*}(t)\\
\end{array}}  \le \frac{4}{{{k_2}}}\left( { \kappa_{max} {C_1} + \frac{4}{{3{k_2}}}{C_2}} \right) +
\| e^{At}Z_0\|, \\
 {\begin{array}{*{20}{c}}
  \vert \eta\vert^{*}(t)  \\
\end{array}}  \le  { \kappa_{max} {C_1} + \frac{{16}}{{3{k_2}}}{C_2}} +\| e^{At}Z_0\|, \\
\end{array} \right.\label{toto1}
\end{gathered}
\end{array}
\end{eqnarray}
where $Z_0$ is the initial condition.

As a consequence, we have, for $t\geq 0$ large enough,

\begin{eqnarray}\label{ISS_limite1}
\begin{array}{rcl}
\begin{gathered}
\left\{ \begin{array}{l}
 {\begin{array}{*{20}{c}}
\vert \xi\vert^{*}(t)\\
\end{array}}  \le \frac{8}{{{k_2}}}\left( { \kappa_{max} {C_1} + \frac{4}{{3{k_2}}}{C_2}} \right), \\
 {\begin{array}{*{20}{c}}
  \vert \eta\vert^{*}(t)  \\
\end{array}}  \le  {{2 \kappa_{max} C_1} + \frac{{32}}{{3{k_2}}}{C_2}}. \\
\end{array} \right.\label{toto3}
\end{gathered}
\end{array}
\end{eqnarray}

\end{lemma}

The proof of the above lemma is given in Appendix.


From the argument of Lemma \ref{l_sup_limite_non_satur}, other ISS bounds for $\xi$ and $\eta$ can simply be derived by considering the system $(S_k)$ defined in  \eqref{integrators} with the controls $\nu_1$ and $\nu_2$ given in \eqref{eq:nu00}.

\begin{lemma}\label{le:ISSnew}
Let $(S_k)$ be the system defined in  \eqref{integrators} with the controls $\nu_1$ and $\nu_2$ given in \eqref{eq:nu00}. Assume that
$$
\Vert y_1\Vert^{*}< \kappa_{max} , \quad \Vert y_2\Vert^{*}<1.
$$
Then the bounds \eqref{ISS_limite1} can be improved as follows: there exists $T>0$ such that, for every $t>T$,

\begin{eqnarray}\label{ISS_limite2}
\begin{array}{rcl}
\begin{gathered}
\left\{ \begin{array}{l}
 {\begin{array}{*{20}{c}}
\vert \xi\vert^{*}(t)\\
\end{array}}  \le \frac{8}{{{k_2}}}\left( {{C_1\vert y_1\vert^{*}(t)} + \frac{4}{{3{k_2}}}{C_2\vert y_2\vert^{*}(t)}} \right), \\
 {\begin{array}{*{20}{c}}
  \vert \eta\vert^{*}(t)  \\
\end{array}}  \le  {{2C_1\vert y_1\vert^{*}(t)} + \frac{{32}}{{3{k_2}}}{C_2\vert y_2\vert^{*}(t)}}. \\
\end{array} \right.\label{toto4}
\end{gathered}
\end{array}
\end{eqnarray}
\end{lemma}
\begin{proof}
The argument is straightforward by replacing  $C_1$ and $C_2$, which were used to bound
$\nu_1$ and $\nu_2$ in \eqref{eq:mu} by $ \kappa_{max} C_1\vert y_1\vert^{*}(t) $ and $C_2\vert y_2\vert^{*}(t) $.
\end{proof}

For the rest of the paper, we choose $C_1,C_2<<1$ so that the limsup of both $\xi$ and $\eta$ are very small. In the subsequent computations, we can assume with no loss of generality that $|\cos\xi-1|\leq \frac{\xi^2}2$, $|\sin\xi|\leq |\xi|$ and $|\sin\xi-\xi|\leq \frac{|\xi|^3}3$.

\begin{proposition}
The following inequality holds true
\begin{eqnarray}
\begin{array}{rcl}
\begin{gathered}
\begin{array}{rcl}
 \dot V  &\leq&  - M({\xi^2 + \eta^2}) - {k_1}{C_1}{y_1}\sigma \left( {{y_1}} \right) - {C_2 y_2}\sigma \left( {{y_2}} \right) + M{\Upsilon _L}^2\left( {{{\left( {{\kappa _r} C_1 \sigma \left( {{y_1}} \right)} \right)}^2} + ({C_2 \sigma \left( {{y_2}} \right)}) ^ 2} \right) \\
  &&+\frac{k_1|y_2||\xi|^3}3+\frac{k_1|y_1|\xi^2}2 + \mu \eta {y_1} + |\eta\xi|  - {k_2}{y_2}{\kappa _r}{C_1}\sigma \left( {{y_1}} \right)
  - {k_2}\xi \mu {y_1} + {k_2}\xi^2. \label{lyap2}\\
 \end{array}
\end{gathered}
\end{array}
\end{eqnarray}
\end{proposition}


\subsection{Estimation of $\dot V$ for $y\notin Y_{k_2}$}
In this subsection, we choose the several parameters so that
$\dot V$ verifies \eqref{final} outside the region $Y_{k_2}$, for $k_2$ large enough.
The results are summarized in the next lemma.
\begin{lemma}\label{le:11}
For the choice of parameters defined in Theorem \ref{theo1}, there exists $k_2$ large enough and $T>0$, such that, for every $t>T$, Eq. \eqref{final} is verified.
\end{lemma}
The proof of Lemma \ref{le:11} is given in Appendix.\\
In the rest of the paper, the symbol $C$ has been used to represent arbitrary constants, depending only on $\beta$.
\begin{remark}
Notice that inside $Y_{k_2}$, the term $k_2\vert y_1\xi\vert$ cannot be controlled since we only have for that purpose the term $\beta k_2\xi^2+Cy_1^2$.
\end{remark}

\subsection{Final step}
Note that outside $Y_{k_2}$,  for $t$ large enough,  $\dot V\leq -\frac{C}{k_2^4}$. To see that, we proceed as before since either $\vert y_1\vert\geq \frac{C}{k_2^2}$ or $\vert y_2\vert\geq \frac{C}{k_2^{3/2}}$. As a consequence, every trajectory of \eqref{system} must reach $Y_{k_2}$ in finite time. Therefore, along every trajectory of \eqref{system},
the value of $V$ is eventually smaller than $V_{max}$, the maximal value of $V$ over the set
$$
\Vert \xi\Vert\leq \frac{C}{k_2^3}, \quad \Vert \eta\Vert\leq \frac{C}{k_2^2}, \quad
\vert y_1\vert\leq \frac{C}{k_2^2}, \quad \vert y_2\vert\leq \frac{C}{k_2^{3/2}}.
$$
By using \eqref{eq:Venc} and Lemma\ref{l_sup_limite_non_satur}, we get
$$
V_{max}\leq \frac{C}{k_2}.
$$
We deduce by using again \eqref{eq:Venc} that, along every trajectory of \eqref{system} and for $t$ large enough
$$
y_1^2+y_2^2\leq \frac{C}{k_2^3}.
$$
We can then use the improved ISS bounds for $\xi$ and $\eta$ obtained in Lemma   \ref{le:ISSnew}. In particular, one gets that, for $t$ large enough,
$$
\vert \xi\vert^{*}(t)\leq \frac{C}{k_2^{3+3/2}}, \quad \vert \eta\vert^{*}(t)\leq \frac{C}{k_2^{2+3/2}},
$$
In turn, this new bound for $\xi$ allows one to shrink the bounded region $Y_{k_2}$ outside which $\dot V$ verifies \eqref{final}. Indeed, one has to satisfy either \eqref{eq:11} or \eqref{eq:22}, which leads to either  $\vert y_1\vert\geq \frac{C}{k_2^{2+3/2}}$ or  $\vert y_2\vert\geq \frac{C}{k_2^3}$.

Continuing the procedure described before, we construct inductively four sequences of positives numbers $y_{1,n}$,  $y_{2,n}$, $\xi_n$ and $\eta_n$, $n\geq 0$, of upper bounds of
$\Vert y_1\Vert^{*}$, $\Vert y_2\Vert^{*}$,$\Vert \xi\Vert^{*}$ and $\Vert \eta\Vert^{*}$ respectively, such that the following inequalities are verified
$$
\xi_n\leq C\frac{y_{1,n}+y_{2,n}}{k_2^3}, \quad \eta_n\leq C\frac{y_{1,n}+y_{2,n}}{k_2^2},
$$
which are obtained from \eqref{ISS_limite2}, and
$$
y_{1,n+1}\leq C\xi_n, \quad y_{2,n+1}^2\leq Ck_2^2\xi_ny_{1,n},
$$
which are, according to \eqref{eq:11} and \eqref{eq:22}, the equations needed to define, at the $(n+1)$-th step, the bounded region outside which $\dot V$ verifies \eqref{final}. It is simple to prove that, for all non negative integer $n$, we have
$$
y_{1,n+1}+y_{2,n+1}\leq \frac{C}{\sqrt{k_2}}(y_{1,n}+y_{2,n}).
$$
This immediately yields the convergence to zero of trajectories of \eqref{system}.

\begin{remark}
The bootstrap procedure we have used above is clearly an instance of a small gain theorem.
\end{remark}

\end{proof}

\section{Stabilization of the original system}
We have stabilized System \eqref{system} in case there is no saturation on $\dot\eta$. We will now show that for every initial condition, the term inside the outer saturation in  $\dot\eta$
becomes bounded by $1$ for $t$ sufficiently large (i.e. there exists $T>0$ such that, for $t>T$, the thesis holds true). Thus the last two equations of \eqref{system} are given by \eqref{integrators}. To show that, we need an ISS-type of result on the system
\begin{eqnarray}\label{integrators1}
\begin{array}{rcl}
(S_k) \quad \left \{ \begin{gathered}
  \dot \xi  = \eta  +C_1{d_1} ,\hfill \\
  \dot \eta  = -D\sigma( \frac{k_1}D\xi  + \frac{k_2}D\eta  + \frac{C_2}Dd_2),\hfill \\
\end{gathered}  \right.
\end{array}
\end{eqnarray}
where $d_1$ and $d_2$ are amplitude-bounded perturbations which amplitudes are bounded by constants eventually depending on $\kappa_{max}$. We first perform the linear change of variable defined by
$$
X(t)=\frac{k_1}D\xi(\frac{k_2t}{k_1}), \ \ Y(t)=\frac{k_2}D\eta(\frac{k_2t}{k_1}).
$$
A direct computation shows that the dynamics of $(X,Y)$ is given by
\begin{eqnarray}\label{integrators2}
\begin{array}{rcl}
\quad \left \{ \begin{gathered}
  \dot X  = Y  +\frac{k_2C_1}Dd_1 ,\hfill \\
  \dot Y  = -\frac1a\sigma(X+Y  + \frac{C_2}Dd_2).\hfill \\
\end{gathered}  \right.
\end{array}
\end{eqnarray}
Since both $\frac{k_2C_1}D$ and $\frac{C_2}D$ are of the magnitude of $\frac1{k_2D}$, these constants can be chosen arbitrarily small. Then, as a consequence of Theorems $2.5$ and $2.6$ in \cite{COCV2001}, one gets that there exists $C(a)>0$ a positive constant only depending on $a$ so that
\begin{equation}\label{est-Y}
\limsup_{t\rightarrow\infty}(\vert X(t)\vert +\vert Y(t))\vert \leq\frac{C(a)}{k_2D}
(\Vert d_1\Vert_{\infty}+\Vert d_2\Vert_{\infty}).
\end{equation}
Therefore, $\vert X(t)+Y(t)  + \frac{C_2}Dd_2(t)\vert$ becomes strictly less than one for $t$ large enough if $\frac1{k_2D}$ is small enough.

The following lemma  provides bounding conditions on  $u_1$ and $u_2$ that would guarantee that the differential equation given in \eqref{omega} is defined for all times $t\geq 0$.

\begin{lemma}\label{explosion}
For $k_2D$ large enough, the differential equation in $\kappa$ given by \eqref{omega} is defined for all times $t\geq 0$.
\end{lemma}
{\bf Proof of Lemma \ref{explosion}}
After multiplying \eqref{omega} by $\kappa$, one can rewrite as follows,
\begin{equation}\label{eq:kap2}
\frac{d\kappa\dot\kappa}{(1+(d\kappa )^2)^{3/2}}=
\vert \kappa\vert V_x\left[\sign(\kappa)d\eta+
(1-\frac{d\vert \kappa\vert}{\sqrt{1+(\kappa d)^2}})+(d\kappa_r\sign(\kappa)-1)\right].
\end{equation}
The right-hand side of the above inequality is majored by
$$
\vert \kappa\vert V_x(d\vert \eta\vert+d\kappa_{max}-1+ \frac1{1+(d\kappa )^2}).
$$
If we can ensure that
\begin{equation}\label{no-explo1}
d\Vert \eta\Vert_{\infty}<1-d\kappa_{max},
\end{equation}
then this will easily imply that $\kappa$ does not blow up in finite time. Indeed, assuming that \eqref{no-explo1} holds, then there exists $K>0$ (only depending on $d\Vert \eta\Vert_{\infty}$ and $d\kappa_{max}$) such that the right-hand side of \eqref{eq:kap2} becomes negative for $\vert \kappa\vert\geq K$. This readily yields that $\vert \kappa\vert$ becomes strictly less than $K$ in finite time and therefore does not blow up in finite time.

We now show that \eqref{no-explo1} holds true with an appropriate choice of the constants
$k_1,k_2$. Without loss of generality we can assume that $\xi(0)=\eta(0)=0$. In that case, one can replace the $\limsup$'s in the left-hand side of \eqref{est-Y} by $\Vert X\Vert_{\infty}+\Vert Y\Vert_{\infty}$. Since $\frac1{k_2D}$ can be chosen arbitrarily small, \eqref{no-explo1} follows.
\ede

\begin{remark}
The proposed method appears to require global localization of the mobile robot and the desired trajectory at every sampling instant with respect to the world frame, which is usually very difficult to obtain in real applications. However, this does not restrict or limit the application of the presented controller in real life. For example, the position coordinates in $(x,y)$ frame of reference can be obtained by a camera and the angle $\psi$ (the direction of vehicle) by a gyroscope. The position of the target point can then be translated into (p,q) frame of reference given in Equation (5). From here, a finite time differentiator can be used to get $\dot p , \dot q$, and later on the angle $\theta$. A simple exponential (even homogeneous finite time) observer can be used to get $\eta$ as in \cite{mazenc_pvtol}.
\end{remark}

\begin{remark}
Since the proof of the convergence is using a Lyapunov function which is strict
outside $Y_{k_2}$, the results of Theorem \ref{theo1} can easily be extended in case where we have only direct access to $(p_r(\cdot),q_r(\cdot))$. Note that the gradient of the Lyapunov function is linear in its argument and thus, if the states are obtained through observers of differentiators, this will require $C_i$, $i=1,2$ in \eqref{final} to be changed with $C_i+f_i(t)$, where $f_i(\cdot)$ is the observation or estimation error. There exist fixed time convergent differentiators such as \cite{moreno2011}, which ensure that the derivative converges in fixed finite time. On the other hand, if the estimation or convergence is asymptotic, it has been shown in \cite{COCV2001, mazenc_pvtol} that if $f_1,f_2$ tend to zero exponentially, then the controller will also converge and the proof will not change.
\end{remark}

\section{Simulations}
The performance of the presented controller can be seen in the simulation results obtained using the following parameters:

$$
d=2\ m,  \ V_x=5\ m.s^{-1}.
$$

$d\kappa_{max}$ has been chosen much smaller than $1$ in order to emphasize upon significant initial conditions (in particular, $\xi(0)$ close to $\pi$) so that the resultant illustrations highlight our claim. The initial conditions imposed upon the error are
$$
e_p(0)=e_q(0)=10\ m,\  \xi(0)=9\pi/10,\
$$

The parameters of the controller was taken as follow:
\beq
	C_1 = 0.1172,\quad C_2 = 0.5, \quad k_1 = 7500, \quad k_2 = 200, \quad D = 50.
\eeq

Figure \ref{tra1} shows the reference trajectory, the target point and vehicle in a 2D coordinate plane. It can be seen that the system converges to the reference trajectory asymptotically. Once the vehicle converges, the target point and the vehicle follow the trajectory very closely. The convergence can also be seen in the error graph shown in figure \ref{erreurs}, where the initial conditions are also visible.

Figures \ref{com_u} and \ref{com_ro} show the control signals $u$ and $\rho$ respectively. It is clear from these figures that the controller does not attain extremely large values, and is bounded. This is an essential property in real systems, which does not result in impossible control signals when the initial error is very large.

\section{Conclusion}
In this paper, we have addressed the problem of path following using a target point rigidly attached to a car type vehicle, by controlling only the orientation of the vehicle by its angular acceleration. The main idea was to determine a control law using saturation which ensures global stabilization in two steps. The proposed Lyapunov function forces the errors to enter a neighborhood of the origin in finite time. The Lyapunov analysis also shows that by a bootstrap procedure, this neighborhood contracts asymptotically to zero. Simulation results illustrate the GAS performance of the controller.

\begin{figure}[htbp]
 \centering
      \includegraphics[width=10cm]{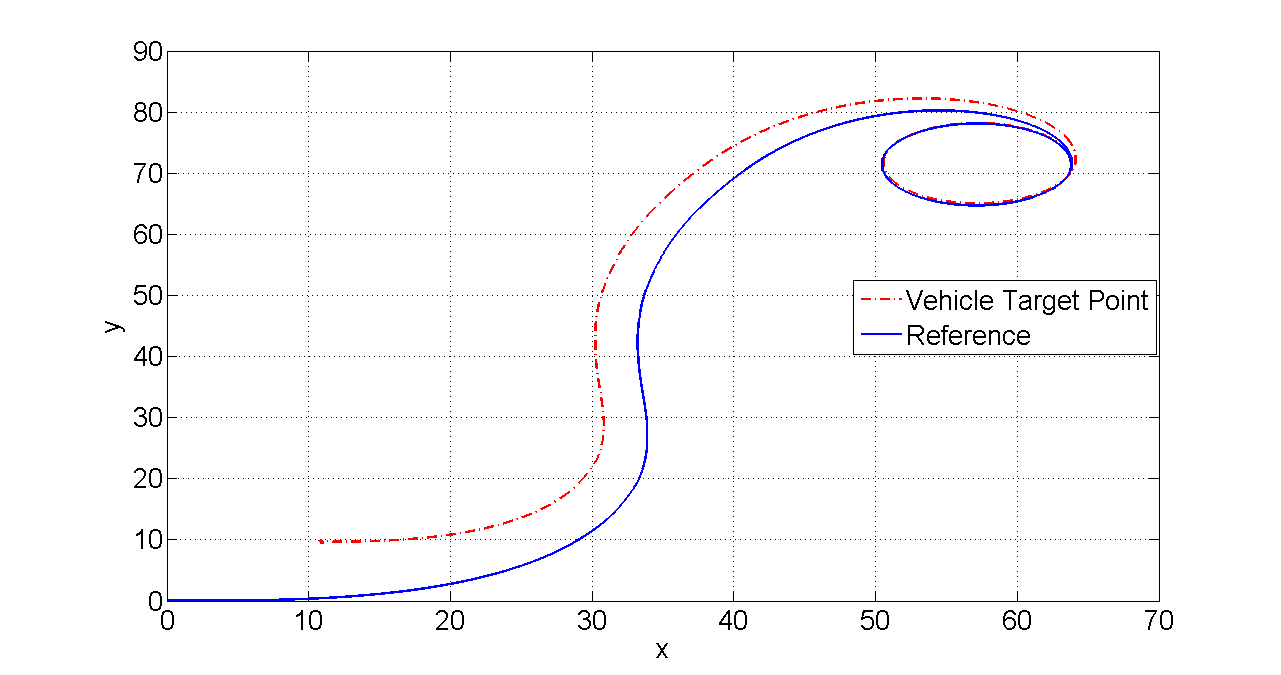}
  \caption{Reference trajectory, of the vehicle and its target point } \label{tra1}
\end{figure}

\begin{figure}[htbp]
 \centering
      \includegraphics[width=10cm]{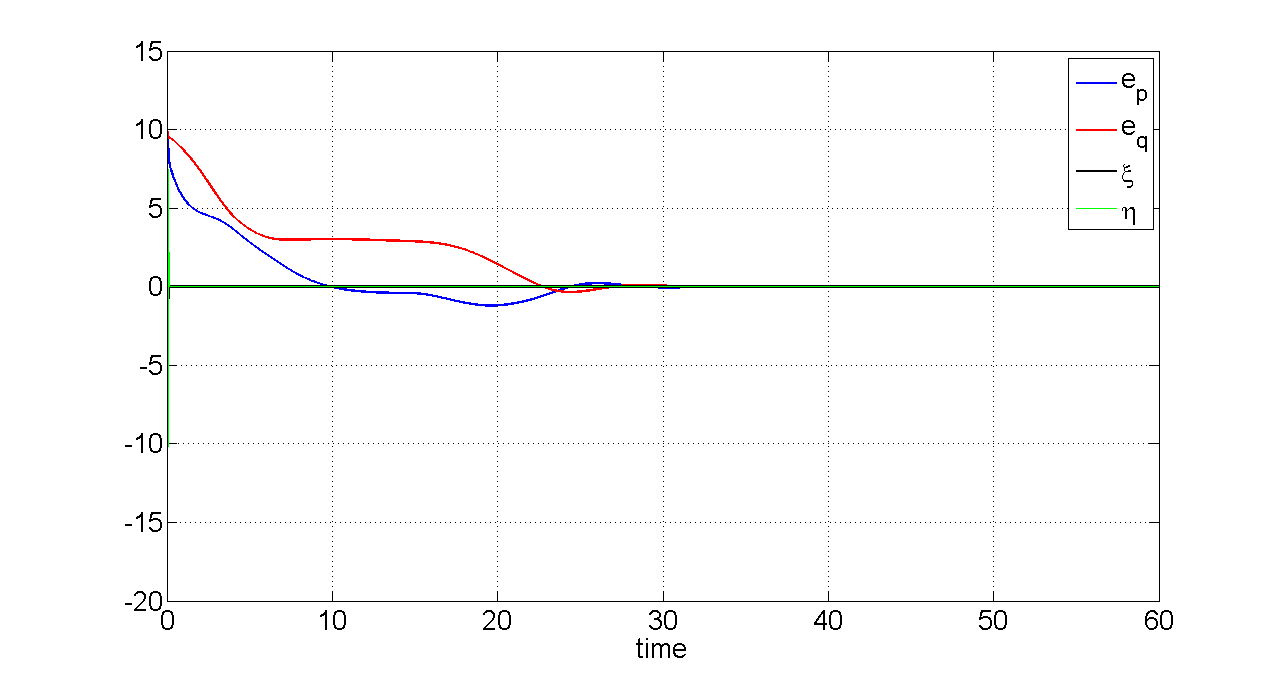}
  \caption{Errors $e_p$, $e_q$, $\xi$ and $\eta$} \label{erreurs}
\end{figure}

\begin{figure}[htbp]
 \centering
      \includegraphics[width=10cm]{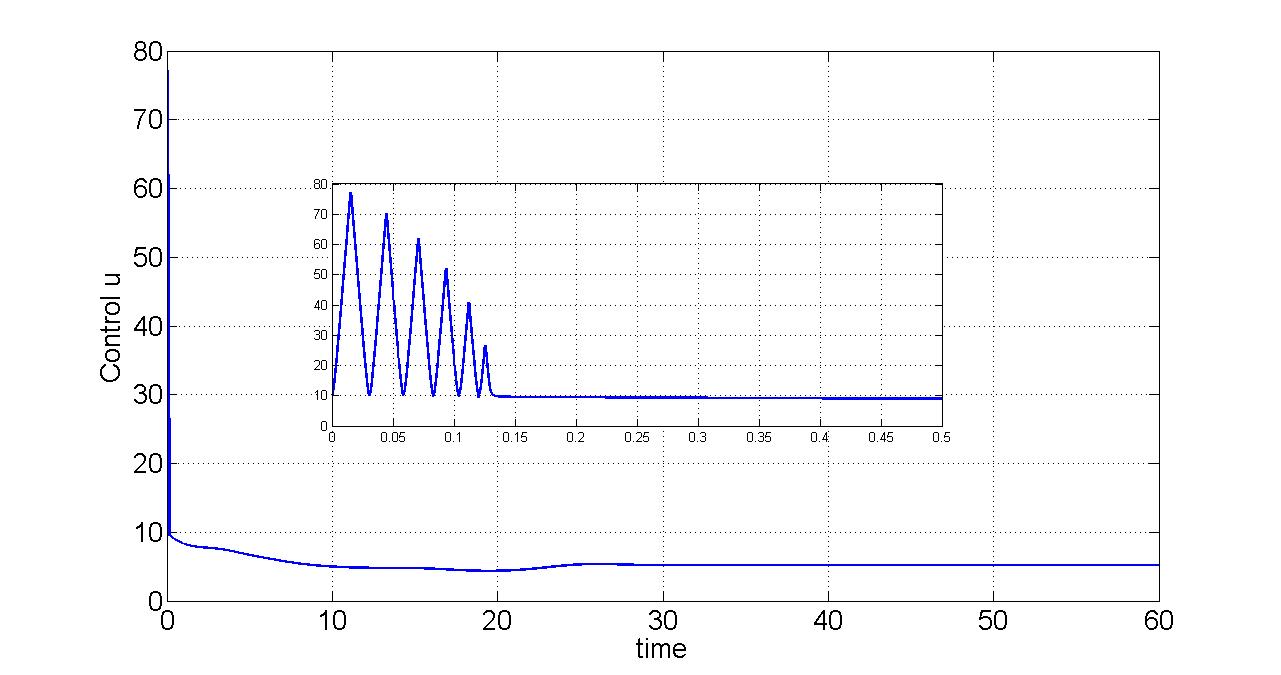}
  \caption{Control $u$} \label{com_u}
\end{figure}

\begin{figure}[htbp]
 \centering
      \includegraphics[width=10cm]{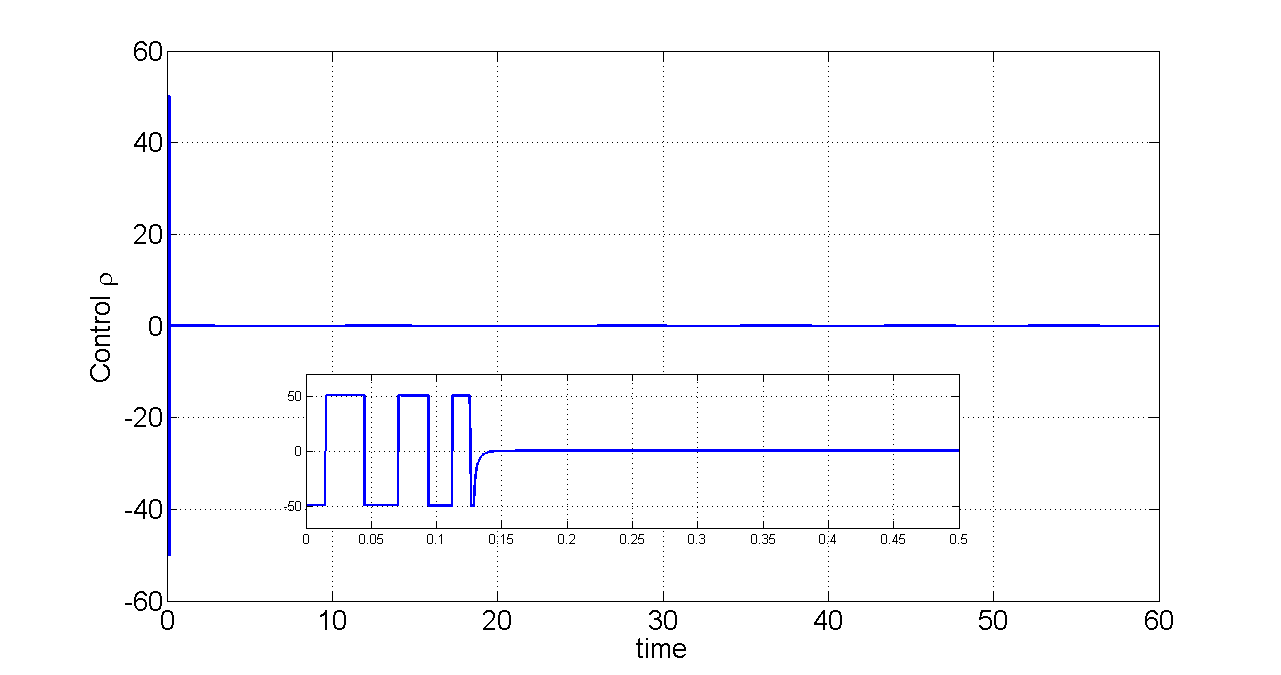}
  \caption{Control $\rho$} \label{com_ro}
\end{figure}

\section{Proof of technical lemmas}
\subsection{Proof of Lemma \ref{le:l2}}
One has
\begin{eqnarray}
\begin{array}{rcl}
\begin{gathered}
G\left( s \right) = {\left( {s{I_2} - A} \right)^{ - 1}},\hfill \\
\end{gathered}
\end{array}
\end{eqnarray}

and the $L_2$-gain is defined by
\begin{eqnarray}
\begin{array}{rcl}
\begin{gathered}
{{\rm{\Upsilon }}_{\rm{L}}} = \mathop {\sup }\nolimits_{\omega  \in {R }} \bar \sigma (G(j\omega )).\hfill \\
\end{gathered}
\end{array}
\end{eqnarray}
where $\bar \sigma$ is the largest singular value of $G(j \omega)$. Since
$$
S(\omega):=G(j\omega )G^*(j\omega )={{\omega ^2}{I_2} + j\omega \left( {A - {A^T}} \right) + {A^T}A},
$$
one has that ${{\rm{\Upsilon }}_{\rm{L}}}^2$ is the inverse of the smallest eigenvalue of $S(\omega)$.

We start the calculation of the matrix $S(\omega)$ and get
 \begin{eqnarray}
\begin{array}{rcl}
\begin{gathered}
S = \left[ {\begin{array}{*{20}{c}}
   {k_1^2 + {\omega ^2}} & {{k_1}{k_2} + j\omega \left( {1 + {k_1}} \right)}  \\
   {{k_1}{k_2} - j\omega \left( {1 + {k_1}} \right)} & {1 + k_2^2 + {\omega ^2}}  \\
\end{array}} \right].
\end{gathered}
\end{array}
\end{eqnarray}

The minimum eigenvalue is equal to
\begin{eqnarray}
\begin{array}{rcl}
\begin{gathered}
\lambda _{min}( \omega ) = \frac{{1 + k_1^2 + k_2^2 + 2{\omega ^2} - \sqrt {{{\left( {1 + k_1^2 + k_2^2 + 2{\omega ^2}} \right)}^2} - 4\left( {{\omega ^2}k_2^2 + {{\left( {{\omega ^2} - {k_1}} \right)}^2}} \right)} }}{2}.
\end{gathered}
\end{array}
\end{eqnarray}

The minimum of $\lambda _{min}$ with respect to $\omega$ is equal to
\begin{eqnarray}
\begin{array}{rcl}
\begin{gathered}
	\lambda _{min} = \frac{{1 + k_1^2 + k_2^2}}{2} - \sqrt {\frac{{{(1 - k_1^2)}^2 + k_2^2(k_2^2 + 2 + 2k_1^2)}}{4}}.
\end{gathered}
\end{array}
\end{eqnarray}
We deduce that
\begin{equation}\label{l2g}
 \Upsilon _L^2 =\frac12+\frac{1+k_2^2}{2k_1^2}+\frac12\sqrt{1+(\frac{2k_2}{k_1})^2+\Big(\frac{2(1+k_2^2)}{k_1^2}\Big)^2}>1.
\end{equation}
If we tune ${k_1} = \frac{3}{{16}}k_2^2$ and $k_2\geq 20$, then
$0.93 < \lambda _{min} < 1$, and $1 < \Upsilon _L < 1.2$.

\subsection{Proof of Lemma \ref{le:P_k}}
 In that case, and by using the Taylor expansion of equation (\ref{l2g}), we have the following asymptotic expansion of $\Upsilon _L^2$,
\begin{equation}\label{eq:exp_k_2}
 \Upsilon _L^2 =1+\frac{3}{2a^2k_2^2}+\frac{1}{2a^2 k_2^4}+O(\frac{1}{k_2^6}),
 \end{equation}

 Then, the Riccati equation proposed in \eqref{eq:ricatti} takes the following form
\begin{equation}\label{eq:P^2}
 \big(\frac{P_k}{ \Upsilon _L}+\Upsilon _LA)^T \big(\frac{P_k}{ \Upsilon _L}+\Upsilon _LA)=S, \hbox{ where }
 S=-I+ \Upsilon _L^2A^TA.
\end{equation}

It can easily be checked that $S$ is definite positive since $\det(S) > 0$:

\[S = \left[ {\begin{array}{*{20}{c}}
   { - 1 + \Upsilon _L^2{a^2}k_2^4} & {\Upsilon _L^2ak_2^3}  \\
   {\Upsilon _L^2ak_2^3} & { - 1 + \Upsilon _L^2\left( {1 + k_2^2} \right)}  \\
\end{array}} \right] , \]

and we get:
$$
\det(S)= \Upsilon _L^2k_2^2(a^2k_2^2(\Upsilon _L^2-1)-1)-(\Upsilon _L^2-1),
$$
and we deduce from \eqref{l2g} the following Taylor expansion for $\det(S)$,

$$
\det \left( S \right) = \frac{1}{2}\Upsilon _L^2k_2^2\big(1+\frac1{k_2^2} + O\left( \frac1{k_2^4}  \right)\big).
$$

Then, the Riccati equation (\ref{eq:P^2}), takes the form:
\[{X^T}X = S,\]
where $X=\frac{P_k}{ \Upsilon _L}+\Upsilon _LA$, and the solution is:
\begin{equation}\label{eq:X}
X=\frac{P_k}{ \Upsilon _L}+\Upsilon _LA=R_\phi\sqrt{S},
\end{equation}
where $R_\phi$ is a rotation of angle $\phi$ and $\sqrt{S}$ is the unique symmetric positive definite matrix whose square is equal to $S$.
We first estimate $\sqrt{S}$ and then $\phi$.

We clearly have
$$
S=bb^T+\gamma e_2e_2^T,
$$
where $b=\begin{pmatrix}\alpha\\ \beta\end{pmatrix}$, $e_2=\begin{pmatrix}0\\ 1\end{pmatrix}$ and
$$
\alpha=\sqrt{(\Upsilon_Lak_2^2)^2-1},\quad \beta=\frac{\Upsilon_L^2 a k_2^3}{\alpha},
\quad \gamma=\frac{\det(S)}{\alpha^2}.
$$

The asymptotic expansions of the above quantities are
$$
\begin{pmatrix}\alpha\\ \beta\end{pmatrix}=\Upsilon_Lk_2\begin{pmatrix}ak_2\big(1-\frac1{2\Upsilon_L^2a^2k_2^4}+O(\frac1{k_2^8})\big)\\ 1+\frac1{2\Upsilon_L^2a^2k_2^4}+O(\frac1{k_2^8})\end{pmatrix},
\quad \gamma=\frac1{2a^2k_2^2}(1+\frac1{k_2^2}+O(\frac1{k_2^4})).
$$
We also need the asymptotic expansions of the eigenvalues of $S$. Since
\[S = \left[ {\begin{array}{*{20}{c}}
   {{\alpha ^2}} & {\alpha \beta }  \\
   {\alpha \beta } & {{\beta ^2} + \gamma }  \\
\end{array}} \right],\]
then,

\[{\lambda _{1,2}} = \frac{\alpha ^2+\beta ^2+\gamma \pm \sqrt {\left( \alpha ^2 + \beta ^2 + \gamma  \right)^2 -4\alpha^2\gamma}}{2}\]
We immediately deduce the following asymptotic expansions for the eigenvalues of $S$,
\begin{equation}\label{eq:lam}
\lambda_1=\alpha^2+\beta^2+O(\frac1{k_2^4}),\quad \lambda_2=\gamma(1+O(\frac1{k_2^2})).
\end{equation}

We use $\bb$ to denote the unit vector $b/\Vert b\Vert$ and define the angle $\phi_b$
so that $R_{\phi_b}e_1=\bb$, where $e_1=\begin{pmatrix}1\\0\end{pmatrix}$.

We have
 $$
 \Vert b\Vert=\Upsilon _Lak_2^2\big(1+\frac1{2a^2k_2^2}+O(\frac1{k_2^4})\big),
 \qquad \bb=\begin{pmatrix}1-\frac1{2a^2k_2^2}+O(\frac1{k_2^4})\\ \frac1{ak_2}\big(1+\frac1{2a^2k_2^2}+O(\frac1{k_2^4})\big)\end{pmatrix},
 $$
 and
 $$
 R_{\phi_b}=l_2Id_2+l_1A_0,
$$
where $A_0=\begin{pmatrix}0&-1\\1&0\end{pmatrix}$ and
$$
  l_1=\frac1{ak_2}\big(1-\frac3{2a^2k_2^2}+O(\frac1{k_2^4})\big), \qquad
  l_2=(1-\frac1{2a^2k_2^2}+O(\frac1{k_2^4})).
  $$
We then get that $ R_{-\phi_b}e_2=l_2e_2-l_1e_1$. We can therefore write

$$
R_{-\phi_b}SR_{\phi_b}=\Vert b\Vert^2e_1e_1^T+\gamma R_{-\phi_b}e_2e_2^TR_{\phi_b},
$$
and deduce that
\begin{equation}\label{eq:S}
R_{-\phi_b}SR_{\phi_b}=(\Vert b\Vert^2+\gamma l_1^2)e_1e_1^T+\gamma l_2^2 e_2e_2^T
-\gamma l_1l_2(e_1e_2^T+e_2e_1^T).
\end{equation}
Finally, we seek a formula of the type
\begin{equation}\label{eq:g0}
R_{-\phi_b}\sqrt{S}R_{\phi_b}=s_1 e_1e_1^T+s_2 e_2e_2^T-s_3(e_1e_2^T+e_2e_1^T),
\end{equation}
where the $s_I$'s are positive. A simple identification leads to the equations
$$
s_1^2+s_3^2=\Vert b\Vert^2+\gamma l_1^2,
\qquad s_2^2+s_3^2=\gamma l_2^2,
\qquad s_3=\frac{\gamma l_1l_2}{Tr(\sqrt{S})}.
$$
We deduce at once from the asymptotic expansions of the eigenvalues of $S$ obtained in \eqref{eq:lam} that
\begin{equation}\label{eq:g}
s_1=\Vert b\Vert(1+O(\frac1{k_2^8})),\qquad s_2=\sqrt{\gamma}(1+O(\frac1{k_2^2})),
\qquad s_3=O(\frac1{k_2^5}).
\end{equation}

From the expression $\Upsilon _L(A-A^T)=R_\phi\sqrt{S}-\sqrt{S}R_{-\phi}$, we obtain
\[{\Upsilon _L}\left( {1 + ak_2^2} \right)\left[ {\begin{array}{*{20}{c}}
   0 & 1  \\
   { - 1} & 0  \\
\end{array}} \right] =  - Tr\left( {\sqrt S } \right)\sin \left( \phi  \right)\left[ {\begin{array}{*{20}{c}}
   0 & 1  \\
   { - 1} & 0  \\
\end{array}} \right] , \]

and we deduce that
$\Upsilon_L(1+ak_2^2)=-Tr(\sqrt{S})\sin(\phi)$, i.e.,

$$
\sin(\phi)=-\big(1-\frac{C_0^2}{2k_2^2}+O(\frac1{k_2^3})\big),
$$
where $C_0=\frac{\sqrt{1-2a}}{a}$. It implies that
$$
R_{\phi}=\frac{C_0}{k_2}\big(1+O(\frac1{k_2})\big)Id_2-\big(1-\frac{C_0^2}{2k_2^2}+O(\frac1{k_2^3})\big)A_0.
$$

We now collect the result of the equation (\ref{eq:X}) and (\ref{eq:g}) to get,
$$
\frac{P_k}{ \Upsilon _L}=-\Upsilon _L A+ R_{\phi}R_{\phi_b}
\Big[s_1 e_1e_1^T+s_2 e_2e_2^T+O(\frac1{k_2^5}))\Big]R_{-\phi_b}.
$$
After a long but straightforward computation, we arrive at \eqref{eq:P}.

\subsection{Proof of Proposition \ref{le:V}}
\begin{proof}
The proof is developed using Young inequality,
$|\eta \xi| \le \frac{\epsilon\eta^2}{2} +\frac{\xi^2}{2\epsilon}$,
where $\epsilon$ is an arbitrary positive constant.

First of all, according to Lemma \ref{le:P_k}, for large $k_2$, the quadratic form $V_k(\xi,\eta)$ satisfies the following inequality
$$
	V_k(\xi,\eta)\geq F_1 k_2 \xi^2 + 2 F_2 \xi \eta + F_3 \frac{\eta^2}{k_2}-\frac{C}{k_2}(\xi^2+\frac{\eta^2}{k_2^2}),
$$
for some positive universal constant $C>0$. Then, by setting $X:=\sqrt{k_2F_1}\eta$ and
$Y:=\frac{\eta}{\sqrt{k_2F_3}}$, we obtain
$$
V_k(\xi,\eta)\geq (1-\frac{C}{k_2^2})X^2+2\frac{F_2}{\sqrt{F_1F_3}}XY+(1-\frac{C}{k_2^2})Y^2.
$$
Since $\frac{F_2^2}{F_1F_3}<1$, the above inequality ensures, for $k_2$ large enough, the existence of $l>0$ only dependent of $a$ such that
$$
V_k(\xi,\eta)\geq l \left(k_2 \xi^2 + \frac{\eta^2}{k_2}\right).
$$

The previous inequality with $M=\beta k_2$ and $k_1= a k_2^2$ computed in \eqref{eq:V} implies
\beq
	V(y_1,y_2,\xi,\eta) &\ge& l \beta k_2^2 \xi^2 + l \beta \eta^2 + \frac{a k_2^2}{2} \left(y_1^2 + y_2^2 \right) - \left| \eta y_2 \right| - \left| k_2 y_2 \xi \right|.
\eeq
By using Young's inequality, we get
\beq
	\left| \eta y_2 \right| &\le& l \beta \frac{\eta^2}{2} + \frac{1}{l\beta} \frac{y_2^2}{2},\\
	\left| k_2 y_2 \xi \right| &\le& l \beta \frac{k_2^2 \xi^2}{2} + \frac{1}{l \beta} \frac{y_2^2}{2}.
\eeq
Which implies,
\beq
	V \ge  \left( \frac{a k_2^2}{2} + \frac1{l \beta} \right) y_2^2 + \frac{l \beta k_2^2}{2} \xi^2 + \frac{l \beta}{2} \eta^2 + \frac{a k_2^2}{2} y_1^2.
\eeq
then for large enough $k_2$, $V$ is a positive quadratic form in $(\xi,\eta,y_1,y_2)$.
\end{proof}

\subsection{Proof of Lemma \ref{l_sup_limite_non_satur}}
The solution of the equation (\ref{matrix_form}) is
\begin{eqnarray}
\begin{array}{rcl}
\begin{gathered}
Z\left( t \right) =\int_0^t {e^{A\left( {t - s} \right)}}U\left( s \right)ds + {e^{At{Z_0}}},\label{toto0}
\end{gathered}
\end{array}
\end{eqnarray}
where $Z_0$ is the initial value of $Z$ for $t=0$. We start by diagonalization of the matrix $A$,
whose eigenvalues are equal to
\begin{eqnarray}
\begin{array}{rcl}
\begin{gathered}
{\lambda _ + } =  - \frac{{{k_2}}}{2} + \frac{1}{2}\sqrt {k_2^2 - 4{k_1}} ,\\
{\lambda _ - } =  - \frac{{{k_2}}}{2} - \frac{1}{2}\sqrt {k_2^2 - 4{k_1}} ,\\
\end{gathered}
\end{array}
\end{eqnarray}
with corresponding eigenvectors
 ${V_ + } = \left( {\begin{array}{*{20}{c}}
   1  \\
   {{\lambda _ + }}  \\
\end{array}} \right)$
and
${V_ - } = \left( {\begin{array}{*{20}{c}}
   1  \\
   {{\lambda _ - }}  \\
\end{array}} \right)$.
From here, we obtain
$A = PD{P^{ - 1}}$, where
\begin{eqnarray}
\begin{array}{rcl}
\begin{gathered}
{\kern 1pt} D = \left( {\begin{array}{*{20}{c}}
   {{\lambda _ + }} & 0  \\
   0 & {{\lambda _ - }}  \\
\end{array}} \right);{\kern 1pt} P = \left( {\begin{array}{*{20}{c}}
   {{V_ + }} & {{V_ - }}  \\
\end{array}} \right) = \left( {\begin{array}{*{20}{c}}
   {\begin{array}{*{20}{c}}
   1  \\
   {{\lambda _ + }}  \\
\end{array}} & {\begin{array}{*{20}{c}}
   1  \\
   {{\lambda _ - }}  \\
\end{array}}  \\
\end{array}} \right);
\end{gathered}
\end{array}
\end{eqnarray}
We get
\begin{eqnarray}\label{eq:mu}
\begin{array}{l}
 {e^{A\left( {t - s} \right)}}U\left( s \right) 
  = \frac{1}{{{\lambda _ - } - {\lambda _ + }}}\left( {\begin{array}{*{20}{c}}
   {\left( {{\lambda _ - }{v_1} - {v_2}} \right){e^{{\lambda _ + }\left( {t - s} \right)}} + \left( {{\lambda _ + }{v_1} + {v_2}} \right){e^{{\lambda _ - }\left( {t - s} \right)}}}  \\
   {{\lambda _ + }\left( {{\lambda _ - }{v_1} - {v_2}} \right){e^{{\lambda _ + }\left( {t - s} \right)}} + {\lambda _ - }\left( {{\lambda _ + }{v_1} + {v_2}} \right){e^{{\lambda _ - }\left( {t - s} \right)}}}  \\
\end{array}} \right) \\
 \end{array}
\end{eqnarray}
The control variables $\nu_1$ and $\nu_2$ are bounded respectively by $ \kappa_{max} C_1$ and $C_2$ and we obtain that the components of the vector $\int_0^t {e^{A\left( {t - s} \right)}}U\left( s \right)ds $ are bounded componentwise by
\begin{eqnarray}
\begin{array}{rcl}
\begin{gathered}
 \left( \begin{array}{c}
 \begin{array}{*{20}{c}}
   {\frac{1}{{\left| {\begin{array}{*{20}{c}}
   {{\lambda _ + }}  \\
\end{array}} \right|}}\left( { \kappa_{max} {C_1} + \frac{{{C_2}}}{{\left| {{\lambda _ - }} \right|}}} \right)}  \\
\end{array} \\
  \kappa_{max} {C_1} + \frac{{{C_2}}}{{\left| {{\lambda _ + }} \right|}} + \frac{{{C_2}}}{{\left| {{\lambda _ - }} \right|.}} \label{toto12}\\
 \end{array} \right).
\end{gathered}
\end{array}
\end{eqnarray}

By taking $k_1 = \frac{3}{16}k_2^2$, we get
$
 {\lambda _ + } =  - \frac{{{k_2}}}{4}$ and ${\lambda _ - } =  - \frac{{3{k_2}}}{4}$.
We can bound the components of the vector defined in \eqref{toto} by
 \begin{eqnarray}\label{intg_inq}
\begin{array}{rcl}
\begin{gathered}
 \left( \begin{array}{c}
 \begin{array}{*{20}{c}}
   {\frac{4}{k_2}\left( { \kappa_{max} {C_1} + \frac{4}{{3{k_2}}}{C_2}} \right)}  \\
\end{array} \\
  \kappa_{max} {C_1} + \frac{{16}}{{3{k_2}}}{C_2} \\
 \end{array} \right).\label{toto2}
 \end{gathered}
\end{array}
\end{eqnarray}
Equation \eqref{toto1} can be deduced directly from Equations \eqref{toto0} and \eqref{toto2}. Moreover, since $A$ is Hurwitz,
we arrive to \eqref{ISS_limite1}.
\subsection{Proof of Lemma \ref{le:11}}
The terms $|\eta\xi| $ and $ {k_2}\xi^2$ are clearly dominated by $\frac{M}4({\xi^2 + \eta^2})$.

In \eqref{lyap2}, one has $M\big(\Upsilon _L\kappa _r C_1\sigma(y_1) \big)^2\leq CM\big(C_1\sigma(y_1) \big)^2$, which is dominated by $\frac{k_1C_1}{4}y_1\sigma(y_1)$
if  $C\beta C_1\leq \frac{ak_2}{4}$. The latter clearly holds true for $k_2$ large enough.

We have $M\big(\Upsilon _L C_2\sigma(y_2) \big)^2\leq CM\big(C_2\sigma(y_2) \big)^2$, which is dominated by $C_2y_2\sigma(y_2)/4$ if $AMC_2^2\leq C_2/4$. The latter is true according to the choice of $C_2$ in \eqref{eq:param}.

We now turn to the control of the term $k_2C_1\vert y_2\kappa _r\sigma (y_1)\vert$ by $\frac{k_1C_1}{4}y_1\sigma(y_1)+\frac{C_2}{4}y_2\sigma(y_2)$. If $\vert y_2\vert\geq 1$, the second term is in control  if $\frac{C_2}4\geq Ck_2C_1$ which holds true. Assume now that $\vert y_2\vert\leq 1$. In the case where $\vert y_1\vert\geq 1$, the first term is in control  if $\frac{k_1C_1}4\geq Ck_2C_1$ which obviously holds true. It remains the case where  $\vert y_1\vert\leq 1$. It is immediate to check that the quadratic form
$\frac{k_1C_1}4y_1^2+\frac{C_2}4y_2^2-Ck_2C_1y_1y_2$ is definite positive.

We next consider the term $k_1\vert y_2\vert \vert \xi\vert^3$. To control it, we first bound $\vert \xi\vert^2$ by $\frac{C}{k_2^6}$ for $t$ large enough. In case $\vert y_2\vert\geq 1$,
then the term is immediately dominated by $\frac{C_2}4\vert y_2\vert $. Otherwise, one has, for $k_2$ large enough,
$$
k_1\vert y_2\vert \vert \xi\vert^3\leq C\frac{\vert y_2\vert}{k_2}\frac{\vert \xi\vert}{k_2^4}\leq C(\frac{y_2^2}{k_2^2}+\frac{\xi^2}{k_2^4}),
$$
the last two terms being controlled by $\frac{C_2}4\vert y_2\vert ^2+ \frac{M}4\xi^2$.

Using again the estimate $\vert \xi\vert$ by $\frac{C}{k_2^3}$  for $t$ large enough, the control of $k_1\vert y_1\vert \vert \xi\vert^2$ reduces to that of $k_2\vert y_1\xi\vert$. It therefore remain to control the latter. This is where we need the hypothesis that $y\notin Y_{k_2}$. Assume first that one wants to get the inequality
\begin{equation}\label{eq:11}
k_2\vert y_1\xi\vert\leq \frac{k_1C_1}{4}y_1\sigma(y_1).
\end{equation}
This holds true if $\vert y_1\vert\geq \frac{C}{k_2^2}$. On the other hand, if one wants to get the inequality
\begin{equation}\label{eq:22}
k_2\vert y_1\xi\vert\leq \frac{C_2}{4}y_2\sigma(y_2),
\end{equation}
it holds if $\vert y_2\vert\geq \frac{C}{k_2^{3/2}}$. In any case, outside $Y_{k_2}$,
one of the two inequalities \eqref{eq:11} or \eqref{eq:22} must hold true and Lemma\ref{le:11} is established.

Finally, with the choice of $M=\beta k_2$ together with \eqref{eq:P}, it is immediate to verfy that $V$ is positive definite. Moreover, we get
\begin{equation}\label{eq:Venc}
a_1k_2\xi^2+\frac{a_2}{k_2}\eta^2+a_3k_2^2(y_1^2+y_2^2)\leq V\leq
d_1k_2\xi^2+\frac{d_2}{k_2}\eta^2+d_3k_2^2(y_1^2+y_2^2),
\end{equation}
for some positive constants $a_i$, $d_i$, $1\leq i\leq 3$.

\bibliography{Car_bib}

\end{document}